\documentclass[a4paper,11pt]{article}

\usepackage{amsmath,amsthm,array}
\usepackage{amssymb}
\usepackage{mathtools}
\usepackage{enumerate}
\usepackage[margin=3cm]{geometry}
 
\usepackage[utf8]{inputenc}
\usepackage{tgtermes}
\usepackage{mathptmx}
\usepackage[T1]{fontenc}
\usepackage{stackrel}
\usepackage{pict2e,picture}

\usepackage{hyperref}
\hypersetup{colorlinks=true}

\makeatletter
\newcommand{\pnrelbar}{%
  \linethickness{\dimen2}%
  \sbox\z@{$\m@th\prec$}%
  \dimen@=1.1\ht\z@
  \begin{picture}(\dimen@,.4ex)
  \roundcap
  \put(0,.2ex){\line(1,0){\dimen@}}
  \put(\dimexpr 0.5\dimen@-.2ex\relax,0){\line(1,1){.4ex}}
  \end{picture}%
}

\newcommand{\precneq}{\mathrel{\vcenter{\hbox{\text{\prec@neq}}}}}
\newcommand{\prec@neq}{%
  \dimen2=\f@size\dimexpr.04pt\relax
  \oalign{%
    \noalign{\kern\dimexpr.2ex-.5\dimen2\relax}
    $\m@th\prec$\cr
    \noalign{\kern-.5\dimen2}
    \hidewidth\pnrelbar\hidewidth\cr
  }%
}
\makeatother

\newtheorem{Theorem}{Theorem}[section]
\newtheorem{Lemma}[Theorem]{Lemma}
\newtheorem{Proposition}[Theorem]{Proposition}
\newtheorem{Corollary}[Theorem]{Corollary}
\newtheorem{Definition}[Theorem]{Definition}

\theoremstyle{definition}
\newtheorem{Example}[Theorem]{Example}

\newtheorem{Remark}[Theorem]{Remark}

\newcommand{\CC}{\mathbb{C}}

\newcommand{\ZZ}{\mathbb{Z}}
\newcommand{\NN}{\mathbb{N}}
\newcommand{\QQ}{\mathbb{Q}}

\newcommand{\Hom}{\operatorname{Hom}}

\newcommand{\Mat}{\operatorname{Mat}}

\newcommand{\Spec}{\operatorname{Spec}}
\newcommand{\ad}{\operatorname{ad}}

\newcommand{\Span}{\operatorname{Span}}
\newcommand{\op }{{\operatorname{o}}}
\newcommand{\Ann}{\operatorname{Ann}}

\newcommand{\Ker}{\operatorname{Ker}}
\newcommand{\Gr}{{\operatorname{Gr}}}

\newcommand{\supp}{\operatorname{supp}}
\newcommand{\lrk}{{\operatorname{lrk}}}
\newcommand{\rrk}{{\operatorname{rrk}}}

\newcommand{\lt}{\operatorname{lt}}
\newcommand{\lm}{\operatorname{lm}}
\newcommand{\lc}{\operatorname{lc}}

\newcommand{\leftbimod}[3]{\vphantom{#1}^{#2}{\kern-#3pt #1}}

\renewcommand{\>}{\rangle} 
     
\numberwithin{equation}{subsection}
  
\title{The graph of a Weyl algebra endomorphism}

\author{Niels Lauritzen and Jesper Funch Thomsen} 
\begin{document}
\maketitle 

\begin{abstract}

Endomorphisms of Weyl algebras are studied
using bimodules. Initially,
for a Weyl algebra over a field of characteristic zero,
Bernstein's inequality implies that holonomic bimodules 
finitely generated from the right or left form 
a monoidal category.

The most
important bimodule in this paper is the graph of an
endomorphism. 
We prove that the graph of an endomorphism of a Weyl algebra
over a field of characteristic zero is a simple bimodule. 
%This result may be viewed in the context of a classical theorem 
%by Burnside on matrix algebras.
The simplicity of the tensor product of the dual graph
and the graph is equivalent to the Dixmier conjecture.

It is also shown how the graph construction leads to a non-commutative 
Gr\"obner basis algorithm for detecting invertibility of
an endomorphism  for Weyl algebras and computing the inverse over arbitrary fields. 
\end{abstract}

\section*{Introduction}

Let $X=\Spec_k(R)$ and
$Y=\Spec_k(S)$ denote affine $k$-schemes, where $k$ is a field and $R$ and $S$
commutative $k$-algebras.
The graph of a $k$-morphism $\varphi:X\rightarrow Y$ is the
closed subscheme in $X\times_k Y\cong \Spec(S\otimes_k R)$ given by
the ideal $J$ generated by
$
a \otimes 1 - 1 \otimes \varphi^*(a)
$
for $a\in S$, where the $k$-algebra homomorphism
$\varphi^*:S\rightarrow R$ defines $\varphi$. The $S\otimes_k R$-module
$S\otimes_k R/J$ identifies with the $R$--$S$ bimodule $R$ with
the natural left multiplication, but with right multiplication
induced by $\varphi^*$.

This observation in commutative affine geometry is the basis of
an interpretation of the graph in the non-commutative case:
the graph of a homomorphism $f: S\rightarrow T$ of general $k$-algebras $S$ and $T$ is 
the $T$--$S$ bimodule $T^f$ defined by
$t x s = t x f(s)$, where $t, x\in T$ and $s\in S$.
The graph has several appealing properties e.g.,~it is finitely generated from the left and a homomorphism
is invertible if and only if its graph is invertible as
a bimodule. 
%In this case
%$$
%\left(A^f\right)^{-1} \cong \Hom_A({}_AA^f, {}_AA)\cong A^{f^{-1}}. 
%$$
  
We first give some general properties of bimodules over Weyl 
algebras. Let $A:=A_n(k)$ 
be a Weyl algebra over a field $k$ of characteristic zero and
$M$  a bimodule over $A$. Then we prove that
$M$ is holonomic as a bimodule i.e., as a left module over
the enveloping algebra $A^e\cong A_{2n}(k)$ if it is 
finitely generated from the left or right. If $N$ is
a holonomic bimodule and $M$ is finitely generated from the right, we
show that $M\otimes_A N$ is a holonomic bimodule. 
Introducing the \emph{left} and \emph{right rank} of a holonomic bimodule
$M$ over $A$, we show that the dual bimodule $\Hom(M, A)$ is holonomic. 
These introductory results are 
basically all proved using Bernstein's inequality.

For an endomorphism $f:A\rightarrow A$ we show by yet
another invocation of Bernstein's inequality that the graph $A^f$  is a simple
bimodule. This prompts a question for general algebras $B$: if $S\subset B$ is a subalgebra
and $B$ is a simple $B$--$S$ bimodule, is $S = B$? If $B$ is 
the matrix algebra over an algebraically closed field, this question 
has an affirmative answer as
a consequence of Burnside's classical theorem. A generalization to infinite
dimensional algebras is not immediate.
For the Weyl algebra $A=A_1(\CC)$, 
there are several examples of (intricate) proper subalgebras $S$, such that
$A$ is a simple $A$--$S$ bimodule. 

The tensor product $\Hom(A^f, A)\otimes_A A^f$ of the dual graph and the graph, is isomorphic to the bimodule $\leftbimod{A^f}{f}{3}$ introduced
by Bavula in \cite{Bavula}. Bavula's result on holonomicity of $\leftbimod{A^f}{f}{3}$ is put
in the framework of tensor products of bimodules. Finite generation of the graph $A^f$ from the right is equivalent to the Dixmier conjecture \cite{LT2}.

Finally, using the graph construction, 
a Gr\"obner basis algorithm emerges for detecting whether an
endomorphism of $A_n(k)$ is an automorphism,
where $k$ is a general field.

In the non-commutative case, $S$-polynomials of differential operators
with relatively prime initial terms do not necessarily reduce to zero.
From the bimodule setting we encounter a sequence of
commuting differential operators in $A_{2n}(k)$ with pairwise relatively prime initial terms.
In this situation reduction to zero from the commutative setting generalizes 
rather easily. To complete the proof of the validity of the algorithm we need that
a surjective endomorphism of a Weyl algebra is an automorphism. This is well
known for fields of characteristic zero. For fields of 
positive characteristic this invokes reduction to
the center.

Even though it is not completely obvious from the current version, this paper was inspired by computer
calculations with Gr\"obner bases in the Weyl algebra.
Experiments with \textit{Macaulay2} \cite{M2} 
originally indicated to us some time ago, that the left ideal 
$(x_2 - Q, \partial_2 + P)$ in the second Weyl algebra 
$A_2(\QQ)$ in the variables $x_1, x_2, \partial_1, \partial_2$ with $P, Q$ in the subalgebra generated by $x_1$ and $\partial_1$, is proper
if and only if $[P, Q] = 1$ (an analogous statement in positive
characteristic is false). It only gradually became clear to to us that the proper framework
for explaining this phenomenon could be found in bimodules and
Bernstein's inequality.

We acknowledge valuable discussions with J.~C.~Jantzen and 
S.~J{\o}ndrup.

\section{Preliminaries}

Let $k$ be a field and $A$ a $k$-algebra.
Our focus in this paper is endomorphisms of $k$-algebras and
for the sake of simplicity we will only consider bimodules over
$A$ in this paper.

\subsection{Bimodules over a ring}

A \emph{bimodule} over $A$ is a $k$-vector space
$M$ with (compatible) left and right $A$-module structures, such that
$$
(a_1 m) a_2 = a_1 (m a_2),
$$
where $m\in M$ and $a_1, a_2\in A$. 
We denote by ${}_AM$, the bimodule
viewed as a left module over $A$ and similarly $M_A$ for $M$ viewed
as a right module over  $A$.
A \emph{homomorphism} $f: M \rightarrow
N$ of bimodules $M$ and $N$ over $A$ is a vector space
homomorphism, which is a homomorphism of left and right modules.

Let $A^\op$ denote the opposite ring to $A$.
Bimodules over $A$ and their homomorphisms form an 
abelian category equivalent to the category of
left modules over the enveloping algebra 
$A^e := A\otimes_k A^\op$. Here a bimodule $M$ over $A$ becomes a
left $A^e$-module through $(a\otimes b) m = a m b$ for $m\in M$ and
$a, b\in A$.
Let $M$ and $N$ be bimodules over $A$. Then
$$
M\otimes N := M_A \otimes_{{}_A} {}_A N
$$
and
$$
\Hom(M, N) := \Hom_A({}_AM, {}_A N) 
$$
are bimodules over $A$ (the latter one through $(a f s)(x) = f(x a) s$). 
A bimodule $P$ over $A$ is called invertible if $-\otimes_A P$ is
an autoequivalence of the category of left $A$-modules. The invertibility
of $P$ is equivalent to the existence of a bimodule $Q$, such that
$$
P\otimes Q \cong A\qquad\text{and}\qquad Q\otimes P \cong A.
$$
In this case $Q\cong \Hom(P, A)$.

For a $k$-algebra endomorphism $f:A\rightarrow A$ and a bimodule $M$ over $A$, we let
$M^f$ denote the $A$-bimodule $M$ with its right module structure twisted by $f$ i.e.,
$a m s = a m f(s)$ for $a, s\in A$ and $m\in M$. Similarly ${}^f M$ denotes
the bimodule $M$ with the left module structure twisted by $f$ i.e., 
$a m s = f(a) m s$ for $a, s\in A$ and $m\in M$. The \emph{graph} of $f$ 
is defined as $A^f$.

\begin{Proposition}\label{Proposition:graphinv}
Let $f:A\rightarrow A$ be an endomorphism and $M$ a bimodule over $A$. Then
\begin{enumerate}[(a)]
\item
$$
M\otimes A^f \cong M^f
$$
\item
$$
\Hom(A^f, A) \cong \leftbimod{A}{f}{3}%\cong {}^f A
$$
\item
$f$ is invertible if and only if $A^f$ is invertible. In this case
$$
A^{f^{-1}}\cong \leftbimod{A}{f}{3}%\cong {}^f A{}^fA.
$$
\item\label{graph:ann}
If $f$ is invertible, then
$$
%\Ann_{A^e} 1_{{}^{f^{-1}}\kern-5pt A} = \Ann_{A^e} 1_{A^f}.
\Ann_{A^e} 1_{\leftbimod{A}{f^{-1}}{3}} = \Ann_{A^e} 1_{A^f}.
$$
\end{enumerate}
\end{Proposition}
\begin{proof}
For $(a)$, the natural map $M\otimes A^f \cong M^f$ given by $m\otimes a\mapsto m a$ is
an isomorphism of bimodules. 
In $(b)$, define $\Hom_A(A^f, A)\rightarrow \leftbimod{A}{f}{2}$ by
$\varphi\mapsto \varphi(1)$. This is an isomorphism of bimodules, since
$(a \varphi s)(1) = \varphi(1 a) s = \varphi(f(a)) s = f(a) \varphi(1) s$.
In $(c)$, $f$ is an isomorphism of bimodules $A^{f^{-1}}\rightarrow {}^f A$.
The identity in $(d)$ follows from
$$
a_1 f(b_1) + \cdots + a_m f(b_m) = 0 \iff
f^{-1}(a_1) b_1 + \cdots + f^{-1}(a_m) b_m = 0,
$$
where $a_1, b_1, \cdots, a_m, b_m\in A$.
\end{proof}

\subsection{The Weyl algebra}

We briefly summarize relevant properties of the Weyl algebra. For proofs and
further details we refer to the monograph \cite[Chapter 1]{Bjork} and the textbook
\cite{Coutinho}.

\subsubsection{Arbitrary characteristic}

Let $k$ be a field of arbitrary characteristic.
The Weyl algebra $A_n(k)$ of order $n$ over $k$ is the free algebra
on $\partial_1, \dots,\partial_n, x_1, \dots, x_n$ with relations
\begin{align}\label{commrules}
\begin{split}
[x_i, x_j] &= 0\\
[\partial_i, \partial_j] &= 0\\
[\partial_i, x_j] &= \delta_{ij}
\end{split}
\end{align}
for $i, j = 1, \dots, n$. 
%We give a proof, independent of the underlying field,  of the following well known result. 
\begin{Proposition}\label{Proposition:wbasis}
The set 
\begin{equation}\label{wbasis}
M=\{x^\alpha \partial^\beta \mid \alpha, \beta\in \NN^n\}
\end{equation}
is a vector space basis of $A_n(k)$ over $k$, where
$x^\alpha = x_1^{\alpha_1} \cdots x_n^{\alpha_n}$ and $\partial^\beta = \partial_1^{\beta_1} \cdots \partial_n^{\beta_n}$ for
$\alpha = (\alpha_1, \dots, \alpha_n)$ and 
$\beta = (\beta_1, \dots, \beta_n)$in $\NN^n$.
\end{Proposition}
\begin{proof}
See \cite[\S2, Lemma 3]{GK}
  % The commutation rules \eqref{commrules} show that $M$ is a generating set for $A_n(k)$. 
% Let $W$ be the vector space over $k$ with basis
% $$
% \{z_{\alpha,\beta} \mid \alpha, \beta\in \NN^n\}.
% $$
% Let $e_1, \dots, e_n$ be the canonical basis vectors in $\NN^n$ and define
% $z_{\alpha,\beta} = 0$ if $\alpha\not\in \NN^n$. Then
% \begin{align*}
% x_i z_{\alpha,\beta} &= z_{\alpha+e_i, \beta}\\
% \partial_j z_{\alpha,\beta} &= z_{\alpha, \beta+e_j} + \alpha_j z_{\alpha-e_j, \beta}
% \end{align*}
% defines a left $A_n(k)$-module structure on $W$ by the relations \eqref{commrules} for 
% $i, j = 1, \dots, n$. Here
% $$
% \left(\sum_{\alpha, \beta\in \NN^n} \lambda_{\alpha,\beta} x^\alpha \partial^\beta\right) z_{0,0} = 
% \sum_{\alpha, \beta\in \NN^n} \lambda_{\alpha,\beta} z_{\alpha, \beta}
% $$
% with $\lambda_{\alpha, \beta}\in k$ and the linear independence of $M$ follows.
\end{proof}

The degree
of \emph{a monomial} $x^\alpha \partial^\beta\in M$ is $|\alpha| + |\beta|$, where
$|\alpha| = \alpha_1 + \cdots + \alpha_n$ and
$|\beta| = \beta_1 + \cdots + \beta_n$
with $\alpha = (\alpha_1, \dots, \alpha_n)$ and $\beta = (\beta_1, \dots, \beta_n)$ both in $\NN^n$.
The \emph{degree}, $\deg(f)$, of $f\in A_n(k)\setminus\{0\}$ is the maximum degree
of the monomials occuring with non-zero coefficient in the expansion of
$f$ in the 
basis \eqref{wbasis}. For $f, g\in A_n(k)\setminus\{0\}$, $\deg(f g) = \deg(f) + \deg(g)$.

The map $\alpha:A_n(k)\otimes_k A_n(k)\rightarrow A_{2n}(k)$
given by 
\begin{align}
\begin{split}\label{double}
\alpha(x_i\otimes 1) &= x_i,\qquad \alpha(1\otimes x_i) = x_{n+i}\\ 
\alpha(\partial_i\otimes 1) &= \partial_i, \qquad\alpha(1\otimes \partial_i) = \partial_{n+i}
\end{split}
\end{align} 
for $i = 1, \dots, n$, is an isomorphism of $k$-algebras.

The map $\tau:A_n(k)\rightarrow A_n(k)$ given
by 
\begin{equation}\label{transp}
\tau(x_i) = x_i,\qquad\tau(\partial_i) = -\partial_i
\end{equation}
for $i = 1, \dots, n$ defines an isomorphism
$A_n(k) \rightarrow A_n(k)^\op$. 
Combined, \eqref{double} and \eqref{transp} give an equivalence
between $A_n(k)$-bimodules and left $A_n(k)^e = A_n(k)\otimes_k A_n(k)^\op \cong A_{2n}(k)$-modules.

\subsubsection{Characteristic zero}

Let $k$ be a field of characteristic zero. The increasing sequence $k = B^n_0 \subset B^n_1 \subset \cdots$
of finite dimensional subspaces given by
$$
B^n_m = \Span_k \left\{x^\alpha \partial^\beta \bigm| \deg(x^\alpha \partial^\beta) \leq m\right\} \subset A_n(k)
$$
is called \emph{the Bernstein filtration} of $A_n(k)$. Clearly $\cup_i B^n_i = A_n(k)$ and $B^n_i B^n_j \subset B^n_{i+j}$ for $i, j\in \NN$. Furthermore, 
$\Gr_B(A_n(k)) = B^n_0 \oplus B^n_1/B^n_0 \oplus \cdots $ is the commutative polynomial ring over $k$ in the $2n$ variables
$[x_1], \dots, [x_n], [\partial_1], \dots, [\partial_n]\in B^n_1/B^n_0$.

A \emph{filtration} $\Gamma$ of a left module $M$ over $A_n(k)$,
is an increasing sequence $\Gamma_0 \subset \Gamma_1 \subset \cdots$ of
finite dimensional subspaces of $M$, such that
$\cup_i \Gamma_i = M$ and 
$B^n_i \Gamma_j \subset \Gamma_{i+j}$ for $i, j\in \NN$. 

Such a filtration is called \emph{good} if  $\Gr_\Gamma (M) = \Gamma_0 \oplus
\Gamma_1/\Gamma_0\oplus \cdots $ is finitely generated as
a $\Gr_B(A_n(k))$-module. 
A left module with a good filtration is finitely
generated. If $M$ is finitely generated by $m_1, \dots, m_r\in M$, then 
$\Gamma_i = B^n_i m_1 + \cdots + B^n_i m_r$ is a good filtration of $M$.

For a good filtration $\Gamma$ of a left module $M$,
there exists a polynomial $p = \frac{e}{d!} x^d + \cdots \in \QQ[x]$ with $e\in \NN$, such that
$$
\dim_k \Gamma_i = p(i)\qquad\text{for $i\gg 0$}.
$$
The degree $d$ and the leading coefficient of this polynomial are 
independent of the good filtration of $M$. For a module $M$ with a
good filtration, we denote $\dim(M):=d$ the \emph{dimension} of $M$ and
$e(M):=e$ the \emph{multiplicity} of $M$. 
The following important result is called 
\emph{Bernstein's inequality}.

\begin{Theorem}[Bernstein]\label{Bernstein}
If $M$ is a finitely generated non-zero left module over $A_n(k)$, then $\dim(M)\geq n$.
\end{Theorem}

A finitely generated left module $M$ over $A_n(k)$ is called \emph{holonomic} if $M = 0$ or $M\neq 0$ and $\dim(M) = n$.
Submodules and quotient modules of holonomic modules are holonomic. If 
$0\rightarrow N \rightarrow M \rightarrow M/N \rightarrow 0$ is a short exact
sequence of holonomic modules, then 
\begin{equation}\label{multadd}
e(M) = e(N) + e(M/N).
\end{equation}

%We have the following important result, also due to Bernstein.

%\begin{Theorem}[Bernstein]\label{Bernstein1}\marginpar{Need Bavula generalizati%on of this}
%Let $M$ be a left module over $A_n(k)$. If $M$ has a filtration $\Gamma$
%with
%$$
%\dim_k \Gamma_v \leq c_1 v^n + c_2(v+1)^{n-1},
%$$
%where $c_1$ and $c_2$ are positive integers, then $M$ is 
%holonomic and $e(M) \leq n! c_1$.
%\end{Theorem}

We have the following analogue \cite[Theorem 2.5]{Bavula}  of a
classical result due to Bernstein \cite[Corollary 1.4]{Bernstein}.

\begin{Theorem}\label{Bavula}
Let $M$ be a left module over $A_n(k)$. If $M$ has a filtration $M_v$
with
$$
\dim_k M_v \leq a v^n + \text{lower order terms}
$$
with $a > 0$, then $M$ is 
holonomic.
\end{Theorem} 
\begin{proof}
Let $N$ be a non-zero finitely generated submodule of $M$ and let
$N_0$ be a finite dimensional generating subspace. 
If $N_0\subset M_j$, then 
$$
N_i := B^n_i N_0 \subset B^n_i M_j \subset M_{i+j}
$$ 
and therefore
$$
\dim_k N_i \leq \dim_k M_{i + j} = a (i + j)^n + \text{lower degree terms}.
$$
Theorem \ref{Bernstein} implies that $N$ is holonomic with multiplicity bounded by $n!\, a$. 
By \eqref{multadd}, this 
implies that $M$ has a maximal finitely generated submodule, which has to be
$M$ itself.
\end{proof}

\section{Bimodules over the Weyl algebra}

In this chapter $k$ denotes a field of characteristic zero. 
A bimodule over the Weyl algebra $A_n(k)$ is called \emph{holonomic} if it is holonomic as a left module over the Weyl algebra $A_{2n}(k)$ through the
isomorphism $A_{2n}(k) \cong A_n(k) \otimes_k A_n(k)^\op$ given
by \eqref{double} and \eqref{transp}.

\subsection{Finite generation from the left or right}

The following theorem underlies the theoretical explanation of several
results in this paper.

\begin{Theorem}\label{Theorem:right}
Let $M$ be a bimodule over $A_n(k)$. 
\begin{enumerate}[(i)]
\item\label{right1}
If 
$M$ is finitely generated as a 
left or right module  over $A_n(k)$, then $M$ is 
holonomic.
\item\label{right2}
If $M$ is cyclic as
a left or right $A$-module, then $M$ is free as a left or right module and
irreducible as a bimodule. 
\end{enumerate}
\end{Theorem}
\begin{proof}
We will give the proof of $\eqref{right1}$ and $\eqref{right2}$ 
with respect to the left module structure.
The proof for the right module structure is similar.

Let $m_1,...,m_r$ denote a set of generators 
of $M$ as a left module over $A_n(k)$.
For $a \in A_n(k)$, 
\begin{equation}\label{expa}
m_i a = \sum_{j=1}^r a_{ji} m_j
\end{equation}
with $a_{ji} \in A_n(k)$ and $i = 1, \dots, r$. With $a$ in \eqref{expa}, 
we put $d(a) = \max\{\deg(a_{ji}) \mid i, j = 1, \dots, r\}$. 
Let $D = \max\{d(a) \mid a = x_1, \dots, x_n, \partial_1, \dots, \partial_n\}$.
Then
\begin{equation}\label{switch}
M_i= \sum_{j=1}^r B_i^{2n} m_j \subseteq \sum_{j=1}^r 
B_{iD}^n m_j,
\end{equation}
and $M_i$ is a good filtration of $M$. Since
\begin{equation}\label{simpl}
\dim_k(M_j) \leq r \cdot \dim_k(B_{j D}^n)
= \frac{r D^{2n}}{(2n!)} j^{2n} + \text{lower degree terms in  $j$},
\end{equation}
it follows by Theorem \ref{Bernstein} that $M$ is holonomic.

For the proof of \eqref{right2}, let $m$ be a generator of $M$ as a left module over $A$. By \eqref{right1} we know that $M$ is a holonomic bimodule.
Consider the surjective map 
$$
f: A_n \rightarrow M
$$
of left modules
given by $f(a) = a m$. Let $J$ denote the kernel  $\Ker(f)$ of $f$. 
With $M_j$ denoting $B_j^{2n} m$
we obtain, using similar arguments as in the proof of \eqref{right1}, that 
$M_j$ is contained in $B_{jD}^n m$, for some integer $D$.
Thus
$$
\dim_k(M_j) \leq  \dim_k(B^n_{j D}\, [1])
$$
where $[1] = 1 + J\in A/J$. If $J\neq 0$, then
$$
\dim_k (B^n_{j D}\, [1]) = c j^d + \text{lower degree terms in  $j$},
$$
with $c>0$ and $d\leq 2n - 1$ (cf.~\cite{Coutinho}, Chapter 9, 3.5 {\sc Corollary}). This contradicts Theorem \ref{Bernstein} for the bimodule $M$. 

If $0\subsetneq N \subsetneq M$ is a proper sub-bimodule, then $M/N$ is a
non-zero cyclic left module with torsion contradicting what we just proved. Therefore $M$ is irreducible as a bimodule if it is cyclic as a left module.
\end{proof}

\subsection{The tensor product}

\begin{Theorem}\label{Theorem:tensor}
Let $M$ and $N$ be bimodules over $A_n(k)$.
If $M$ is finitely generated as a right module and
$N$ is holonomic, then
the bimodule
$$
M \otimes_{A_n(k)} N
$$ 
is holonomic.
\end{Theorem}
\begin{proof}
Choose generators $m_1,m_2,\dots, m_r$ of $M$ as 
a right $A_n(k)$-module. Similar to the proof of 
Theorem \ref{Theorem:right} we may   find an
integer $D$, such that 
\begin{equation}\label{switch2} M_i = \sum_{j=1}^r B_i^{2n} m_j = \sum_{j=1}^r 
m_j B_{iD}^n.
\end{equation}
Let $\{ N_i \}_{i \geq 0}$ denote a good filtration of $N$,
and consider the filtration $\{ T_s \}_{s \geq 0}$
of $T = M \otimes_{A_n(k)}  N$, defined by letting $T_s$
denote the $k$-span of elements of the form 
 $m \otimes n$, with $m \in M_i$ and 
$n \in N_j$, with $i+j \leq s$.
By \eqref{switch2}, $T_s$ is contained in 
the subspace spanned by 
\begin{equation*}
m_j \otimes \widehat{n}
\end{equation*}
where $\widehat{n}\in N_{s d}$. In particular, 
\begin{equation*}
\dim_k(T_s) \leq r  \dim_k(N_{s D}).
\end{equation*}
The right hand side of the inequality is polynomial of degree $2n$ in $s$, and 
$T$ is holonomic by Theorem \ref{Bavula}.
\end{proof}

\subsection{The left and right rank of a bimodule}

\begin{Proposition}
Let $M$ be a holonomic bimodule over $A_n(k)$. If
$m_1, \dots, m_r\in M$ are left or right linearly independent over
$A_n(k)$, then
$$
r \leq e(M).
$$
\end{Proposition}
\begin{proof}
We give the proof for left linearly independent elements. The right
linear independence is similar. Let $m_1,m_2, \dots,m_r$ 
denote elements in $M$ which are linearly independent 
under the left action of $A_n(k)$. As holonomic modules
modules are cyclic, we may   choose an element
$m  \in M$ generating $M$ as a bimodule over $A_n(k)$.

The elements $m, m_1, \dots, m_r\in M$ then generate 
$M$ as a bimodule over $A_{n}(k)$. Therefore
$$
M_i = B^{2n}_i m +  B^{2n}_i m_1 + \cdots + B^{2n}_i m_r
$$
defines a good filtration of $M$.
Using the left action, we define the subspace
$$
N_i = B^n_i m_1 + \cdots + B^n_i m_r \subset M_i.
$$
Here
\begin{equation*}
\dim_k(N_i) = r \dim_k(B_i^n) = 
r \binom{2n+i}{2n} = 
\frac{r}{2n!} i^{2n} + \text{lower
degree terms in $i$ },
\end{equation*}
since $m_1, \dots, m_r$ are left linearly independent. Since $M$ is holonomic and $\dim_k N_i \leq 
\dim_k M_i$, this implies by comparing leading terms that $r\leq e(M)$.
\end{proof}

This leads to the following definition.

\begin{Definition}
Let $M$ be a holonomic bimodule over the Weyl algebra $A_n(k)$.
Then the left (right) rank $\lrk(M)$ $(\rrk(M))$ of $M$ is defined as the
rank of the largest free left (right) submodule in $M$. 
\end{Definition}

\begin{Theorem}
Let $M$ denote a holonomic $A:=A_n(k)$-bimodule.
Then the dual bimodule $\Hom(M, A)$ is finitely generated as a 
right module and holonomic as a bimodule over $A$.
\end{Theorem}
\begin{proof}
Let $m_1, \dots, m_r$ denote a basis of a maximal
free left submodule $F$ of $M$, where $r = \lrk(M)$.
The map
\begin{equation}\label{rmap}
\alpha: \Hom(M ,A) \rightarrow A^r
\end{equation}
given by
\begin{equation*}
\alpha(f)= (f(m_1), \dots, f(m_r))
\end{equation*}
is a homomorphism of right modules over $A$.
We claim that $\alpha$ is injective. If
$\alpha(f) = 0$, then $f$ defines
a homomorphism $\overline{f}: M/F\rightarrow A$.
Since $F$ is a maximal free submodule, $M/F$ 
is a torsion module. This implies that
$\overline{f} = 0$, since $A$ has no zero
divisors. Therefore $f = 0$ and $\alpha$ is
injective.

Finally, $A$ being right noetherian implies by the injectivity of 
$\alpha$ in \eqref{rmap} that
the dual bimodule is finitely generated as a right $A$-module.
Now holonomicity follows from Theorem \ref{Theorem:right}.
\end{proof}

\subsection{The graph of a Weyl algebra  endomorphism}

An immediate application of Theorem \ref{Theorem:right} is the 
following.

\begin{Corollary}\label{Corollary:Graph}
Let $k$ be a field of characteristic zero and let $A = A_n(k)$ denote the $n$-th Weyl 
algebra over $k$. If $f$ is an endomorphism of $A$, then the graph $A^f$ and the dual graph ${}^f A$ are
holonomic and simple bimodules over $A$.
\end{Corollary}
\begin{proof}
This follows from part \eqref{right2} of Theorem \ref{Theorem:right}, since
$A^f$ is cyclic as a left module and ${}^f A$ is cyclic as a
right module.
\end{proof}

We also obtain the following result due to Bavula 
\cite[Corollary 1.4]{Bavula}.

\begin{Corollary}\label{Corollary:doublegraph}
Let $f:A\rightarrow A$ be an endomorphism. Then the bimodule
\begin{equation}\label{bavulamod}
{}^f A^f = {}^f A\otimes_{A} A^f
\end{equation}
is holonomic.
\end{Corollary}
\begin{proof}
This is a consequence of Corollary \ref{Corollary:Graph} and Theorem \ref{Theorem:tensor}.
\end{proof}

A famous question (cf.~\cite{Dixmier}, {\bfseries 11. Probl\'emes}, 11.1) due to Dixmier asks if an endomorphism $f$ of
the first Weyl algebra $A_1(k)$ over a field $k$ of characteristic zero is an
automorphism. This is equivalent to $f(A_1(k)) = A_1(k)$, since an
endomorphism of the Weyl algebra is injective.

One may ask the same question for the $n$-th Weyl algebra. This
is equivalent to the simplicity of the bimodule tensor product of the
two simple modules in \eqref{bavulamod}, since
$f(A)$ is a bisubmodule of ${}^fA^f$
 (cf.~\cite{Bavula}, p.~686).

\subsection{Burnside's theorem on matrix algebras}

Based on the simplicity of the graph (Corollary \ref{Corollary:Graph}) the following question emerges. 

\begin{quote}
\emph{If $S$ is a subalgebra of an algebra $A$ and
$A$ is a simple $A$--$S$ bimodule, is $S = A$?}
\end{quote}

Phrased in this generality, the answer is no.
However, for a matrix algebra $A$ over an algebraically closed field, the question
has a positive answer. This is related to Burnside's classical theorem \cite[(27.4)]{CR} 
that every proper subalgebra of $A$ has an invariant 
subspace.

\begin{Theorem}\label{Theorem:BurnsideMatrix}
Let $k$ be an algebraically closed field and $S$ a subalgebra of 
$A=\Mat_n(k)$. If $A$ is simple as an $A$--$S$ bimodule, 
then $S = A$.
\end{Theorem}
\begin{proof}
Let $v$ be a non-zero row vector in $k^n$ and $f\in A$ the matrix with rows $v$.
Then by assumption $A f S = A$. 
Let $I(v S)$ be the left ideal of matrices in $A$ whose row space is 
contained in the subspace $v S \subset k^n$. 
Clearly $f S \subset I(v S)$. Therefore
$I(v S) = A = I(k^n)$ and $v S = k^n$.
The classical Burnside theorem on matrix algebras
now implies that $S = A$.
\end{proof}

We give three examples showing that this result does not, however, generalize easily to the infinite dimensional case. In each of the examples, $k$ denotes an algebraically closed field of characteristic zero.
Example \ref{Example:qtorus}  was communicated to us by S\o ren J\o ndrup.

\begin{Example}\label{Example:Weyl1}
Let $S$ be the subalgebra of $A = A_1(k)$ given by
$S = k + x A$. It follows from Proposition \ref{Proposition:wbasis}
that $\partial\not\in S$ and therefore $S\neq A$.
For $f\in A\setminus\{0\}$, $A f S \supset A f x A = A$, since $A$
is a simple ring and $x$ is not a zero divisor. Therefore
$S$ is a proper subalgebra of $A$, but $A$ is a
simple $A$--$S$ bimodule.
\end{Example}

\begin{Example}\label{Example:qtorus}
Consider the quantum plane $A_q = k\<x, y\>/ (x y - q x y)$, where 
$q\in k$ is not a root of unity. Localizing $A_q$ in powers of $x$ and $y$ we get the quantum torus
$$
T_q = \left\{\sum_{i,j} \lambda_{ij} x^i y^j\, \middle| \, \lambda_{ij}\in k, i, j\in \ZZ\right\},
$$
which is a simple noetherian ring. The subalgebra $S$ generated by
$x^{-2}, y^{-2}, x^2, y^2$ is the subset of Laurent polynomials in even powers
of $x$ and $y$. This subalgebra is proper and simple. Furthermore, $T_q$ is a simple
$T_q$--$S$ bimodule.
\end{Example}

The subalgebra $S$ in Example \ref{Example:Weyl1} is not simple, since
$x A\subset S$ is a non-trivial two-sided ideal. 
Like the quantum torus in Example \ref{Example:qtorus}, the
Weyl algebra $A$ in Example \ref{Example:Weyl1} also affords proper simple subalgebras $S$,
such that $A$ is a simple $A$--$S$ bimodule. 

\begin{Example}
This example can be traced back to \cite{Dixmier2}.
Let $U$ denote
the enveloping algebra for the semisimple Lie algebra $\mathfrak{sl_2}(k)$.
Recall that $U$ is generated by $E$, $F$ and $H$ with the relations
$$
[E, F] = H,\qquad [H, E] = 2 E \qquad\text{and}\qquad [H, F] = -2 F.
$$
The center of $U$ is generated by the Casimir element $Q = 4 F E + H^2 + 2 H$. It is known that the primitive ideals in $U$ all have the form $I_c = 
U(Q - c)$ for $c\in k$ and that $U/I_c$ is a simple ring for
$c\neq n^2 + 2 n$ with $n = 0, 1, 2, \dots$.
There is
an embedding (cf.~\cite{Dixmier2}, Remarques 7.2)
$
U/I_c\rightarrow A_1(k)
$
given by 
$$
E\mapsto -\mu x - \partial x^2,\quad F\mapsto \partial\quad\text{and}\quad
H\mapsto \mu + 2 \partial x
$$
for $c = \mu^2 + 2\mu$. 
The subalgebra $S = U/I_1$ is a simple subalgebra of $A_1(k)$ 
containing $\partial$.
%Therefore $U/I_1$ embeds into $A_1(k)$ as a
%simple subalgebra containing $\partial$.

In general, if $R$ is a simple subalgebra of $T$ and $T f R \cap R \neq\{0\}$ for $f\in T\setminus \{0\}$, then 
$T f R = T$.  

The fact that $\partial\in S$ implies that 
$A f S \cap S \neq \{0\}$ for $f\in A\setminus \{0\}$, since 
$\ad(\partial)^n(f)\in S\setminus\{0\}$ for some $n\geq 0$. This shows that $A_1(k)$ is a
simple $A_1(k)$--$S$ bimodule. Furthermore, $M = x^{-1} k[x^{-1}]$ is an
$S$-stable subset for the natural action of $A_1(k)$ on
$k[x, x^{-1}]$ corresponding to the Verma module with highest weight $-1$ for 
$\mathfrak{sl_2}(k)$. Since $x M \not\subset M$, $S$ is a proper subalgebra. One may prove that $A$ is not finitely generated as a right or left $S$-module.
Furthermore $A$ is not flat as an $S$-module \cite[7. Remarks, 2.]{Smith}. This
should be compared with the flatness of endomorphisms of Weyl algebras proved in
\cite{LT}.
\end{Example}

%For an endomorphism $f$ of $A_n(k)$, the subalgebra
%$f(A_n(k))$ is a simple finitely generated subalgebra of $A_n(k)$. Therefore a %positive answer to
%Question \ref{Question:Burnside} would imply a positive answer
%to Dixmier's question by Corollary \ref{Corollary:Graph}.

%\begin{Proposition}
%If $M$ is holonomic as $A_n(\CC)$-bimodule, then
%$$
%M\otimes_{A_n(\CC)} A_n(\CC)^f
%$$
%is holonomic, where $f$ is an endomorphism of $A_n(\CC)$.
%\end{Proposition}

\section{A Gr\"obner basis algorithm for the inverse}

Let $A = A_n(k)$ be the $n$-th Weyl algebra over an arbitrary field $k$.
%and $\Gr(A)$ the graded ring of the Bernstein filtration.
Gr\"obner basis theory for the Weyl algebra was initiated by Galligo
\cite{Galligo} and later extended to a wider class of
non-commutative algebras with monomial bases by several authors. 
Here we will focus on the Weyl algebra and
briefly recall the theory, which is rather close to the classical theory
of Gr\"obner bases for  
the commutative polynomial ring in $n$ variables.

\subsection{Gr\"obner basics for the Weyl algebra}

Let $\prec$ be a term order on the monomials $M = \{x^\alpha \partial^\beta \mid
\alpha, \beta\in \NN^n\}\subset A$, i.e.
\begin{align}
&\text{$\prec$ is a total order}\\
&1\prec x^\alpha \partial^\beta\\
&x^\alpha \partial^\beta \prec x^{\alpha'} \partial^{\beta'} \implies
x^{\alpha  + \gamma} \partial^{\beta + \delta} \prec x^{\alpha'+\gamma} \partial^{\beta'+\delta}\label{tomult}
\end{align}
for every $\alpha, \beta, \alpha', \beta', \gamma, \delta\in \NN^n$.
The support, $\supp(f)$, of a differential operator
\begin{equation}\label{fsupp}
f = \sum_{(\alpha, \beta)\in \NN^{2n}} \lambda_{\alpha, \beta} x^\alpha \partial^\beta\in
A\setminus \{0\}
\end{equation}
is the set of monomials $x^\alpha\partial^\beta$ in \eqref{fsupp} with $\lambda_{\alpha, \beta}\neq 0$.
The \emph{leading monomial} $\lm(f)$ of $f$ is the maximal monomial in $\supp(f)$. If $\lm(f) = x^u \partial^v$, then $\lt(f):=\lambda_{u, v} x^u \partial^v$ is
called the \emph{leading term} and $\lc(f) := \lambda_{u,v}$ the
\emph{leading coefficient} of $f$.

Let $\leq$ denote the partial order on $\NN^{2n}$ given by
$(u_1, \dots, u_{2n}) \leq (v_1, \dots, v_{2n})$ 
if and only if
$u_1\leq v_1 \wedge \cdots \wedge u_{2n}\leq v_{2n}$.
This
partial order is transferred to $M$ and reads
$m_1 \leq m_2$ if and only if $m_1$ divides $m_2$ as commutative monomials.

If $\lm(d) = x^{\alpha} \partial^\beta \leq
\lm(f) = x^{\alpha'} \partial^{\beta'}$ for $f, d\in A$, then
\begin{align*}
f - q d &= 0\qquad\text{or}\\
\lm(f - q d) &\precneq \lm(f),
\end{align*}
where
$q = \lc(d)^{-1} \lc(f) x^{\alpha'-\alpha} \partial^{\beta' - \beta}$.
This is basically a consequence of \eqref{tomult} and the identity
\begin{equation}\label{monmult}
(x^{\alpha} \partial^{\beta}) (x^{\alpha'} \partial^{\beta'}) = x^{\alpha + \alpha'} \partial^{\beta + \beta'} + 
\sum_{i\geq 1} c_i x^{\alpha + \alpha' - i} \partial^{\beta + \beta'- i}, \qquad
c_i\in k
\end{equation}
for multiplication of monomials in the first Weyl algebra. 
An element $f\in A\setminus\{0\}$ is called
\emph{reducible} by a finite subset $S\subset A$ if
there exists $m\in \supp(f)$ with $\lm(g) \leq m$ for
some $g\in S\setminus\{0\}$.

These observations allow for the commutative theory to be carried over
almost verbatim i.e., the division algorithm, the $S$-polynomial, 
Buchberger's $S$-criterion, uniqueness of reduced Gr\"obner bases etc.
The computations, however, are severely complicated by the ``non-commutative''
multiplication \eqref{monmult}.

Let $A S$ denote the left 
ideal generated by $S$, where $S\subset A$.

\begin{Definition}
\leavevmode
\begin{enumerate}[(i)]
\item
A finite subset $G\subset A$ is called a Gr\"obner basis (with respect
to $\prec$) if 
for every $f\in A G\setminus \{0\}$, there exists $d\in G$ with
$$
\lm(d) \leq \lm(f).
$$
\item
A Gr\"obner basis $G = \{f_1, \dots, f_r\}$ is called reduced if 
$f_i\neq 0, \lc(f_i) = 1$ and $f_i$ is not reducible by $G\setminus\{f_i\}$ for
$i = 1, \dots, n$.
\item
A Gr\"obner basis $G$ is called a Gr\"obner basis for a left
ideal $I\subset A$ if $I = A G$.
\end{enumerate}
\end{Definition}

\begin{Theorem}
Let $I\subset A$ be a non-zero left ideal and $\prec$ a term order. There
exists a unique reduced Gr\"obner basis for $I$ with respect to
$\prec$.
\end{Theorem}

We briefly recall \S2.9 in \cite{CLO}.

\begin{Definition}
For a finite subset $G = \{g_1, \dots, g_m\}\subset A$, we say that $f\in A$ reduces to zero modulo $G$ (denoted $f\longrightarrow_G 0$) if
$$
f = a_1 g_1 + \cdots + a_m g_m
$$
for $a_1, \dots, a_m\in A$ and $\lm(a_i g_i) \prec \lm(f)$ for every $i = 1, \dots, m$ with $a_i g_i\neq 0$. 
\end{Definition}

The $S$-polynomial of $f, g\in A\setminus\{0\}$ is 
$$
S(f, g) := \lc(g) m_f(g) f - \lc(f) m_g(f) g, 
$$
where $m_p(q)$ is the monomial given by the least common multiple of $\lm(p)$ and $\lm(q)$ divided by
$\lm(p)$ for $p, q\in A\setminus\{0\}$. In particular, 
$\lt(\lc(g) m_f(g) f ) = \lt(\lc(f) m_g(f) g)$.

\begin{Theorem}
A finite subset $G\subset A$ is a Gr\"obner basis if and only if
$S(f, g) \longrightarrow_G 0$ for every $f, g\in G$.
\end{Theorem}

The following simple example illustrates a significant difference from 
commutative Gr\"obner bases.

\begin{Example}\label{Example:bad}
The subset $S = \{\partial, x\}\subset A_1(k)$ is not a Gr\"obner basis, since
$\partial x - x\partial = 1\in A_1(k) S$ and $x\nleq 1, \partial \nleq 1$. 
\end{Example}

The monomials $x$ and $\partial$ in Example \ref{Example:bad} are relatively prime, but
they do not commute. A remedy for this situation is given by the following result.

\begin{Theorem}\label{Theorem:GBlmprime}
A set $G = \{f_1, \dots, f_r\}\subset A$ of 
commuting differential operators
i.e.
$f_i f_j = f_j f_i$ for $i, j = 1, \dots, r$ is
a Gr\"obner basis if their leading monomials,
$\lm(f_i)$ and $\lm(f_j)$ are relatively prime for $i\neq j$. 
\end{Theorem}
\begin{proof}
The proof of Proposition 4 of  \S2.9 in \cite{CLO} carries over
verbatim: to show that  $S(f_i, f_j)\longrightarrow_G 0$ one
only needs $f_i f_j = f_j f_i$ along with 
$\lm(f_i)$ and $\lm(f_j)$ being relatively prime.
%For the $S$-polynomial of $f:=f_i$ and
%$g:=f_j$, the calculation in the commutative case 
%\begin{align*}
%S(f, g) &= m_g f - m_f g\\
%&= (g - q) f - (f - p) g\\
%&=  g f - q f - f g + p g\\
%&= p g - q f
%\end{align*}
%carries over verbatim (see the proof of
%Proposition 4 of  \S2.9 in \cite{CLO}). Again the
%leading monomials of $p g$ and $q f$ are distinct, since
%$\ini_\prec(f)$ and $\ini_\prec(g)$ are relatively prime, and
%therefore $S(f, g) \longrightarrow 0$ modulo $\{f, g\}$.
%By Buchberger's $S$-criterion, $G$ is a Gr\"obner basis.
\end{proof}

\subsection{From the graph to the left ideal}

\begin{Lemma}\label{Lemma:EndSur}
Let $k$ be an arbitrary field and $A = A_n(k)$. A surjective
endomorphism $f: A \rightarrow A$ is an automorphism.
\end{Lemma}
\begin{proof}
If $k$ is a field of characteristic zero, $A$ is a simple ring. Therefore $f$
is injective in this case.
Let $k$ be a field of positive
 characteristic $p>0$. Then the center of $A$ is
the commutative polynomial ring
$$
C = k[x_1^p, \dots, x_n^p, \partial_1^p, \dots, \partial_n^p]
$$
in $2n$ variables 
and $A$ is a free module over $C$ with the two bases 
(cf.~\S3.2 and \S3.3 in \cite{Tsuchimoto2} and the \emph{Nousiainen lemma} in \cite{ML})
\begin{align}
\begin{split}\label{twobases}
&\{x_1^{i_1}\cdots x_n^{i_n} \partial_1^{j_1} \cdots \partial_n^{j_n} \mid
0\leq i_1, \dots, i_n, j_1, \dots, j_n \leq p-1\}\\
&\{f(x_1)^{i_1}\cdots f(x_n)^{i_n} f(\partial_1)^{j_1} \cdots f(\partial_n)^{j_n} \mid
0\leq i_1, \dots, i_n, j_1, \dots, j_n \leq p-1\}.
\end{split}
\end{align}
The endomorphism $f$ restricts to an endomorphism
$\overline{f}: C\rightarrow C$, since $f(C) \subset C$. Applying $f$ in terms
of the two bases in \eqref{twobases} if follows that $\overline{f}$
is surjective, since $f$ is surjective. Therefore $\overline{f}$ and hence
$f$ is an automorphism.
\end{proof}

\begin{Remark}
Let $k$ be a field of positive characteristic.
An endomorphism of the first Weyl algebra $A_1(k)$ is always injective. 
Surprisingly, endomorhisms
of $A_n(k)$  fail to be injective for $n\geq 2$ (cf.~\cite{ML}).
\end{Remark}

\begin{Theorem}\label{Theorem:Ae}
Let $k$ be an arbitrary field and $A = A_n(k)$. For an endomorphism
$$
f: A \rightarrow A,
$$
we define the left ideal
\begin{align*}
J = (&1\otimes x_1 - f(x_1)\otimes 1, \dots, 1\otimes x_n - f(x_n)\otimes 1, \\
&1\otimes \partial_1 - f(\partial_1)\otimes 1, \dots, 1\otimes \partial_n - f(\partial_n)\otimes 1)\subset A^e. 
\end{align*}
This left ideal equals the annihilator $\Ann_{A^e}1_{A^f}$.
\begin{enumerate}[(i)]
\item\label{idn1}
If $f$ is invertible, then
\begin{align*}
J = (&x_1\otimes 1 - 1\otimes f^{-1}(x_1), \dots, x_n\otimes 1 - 1\otimes f^{-1}(x_n), \\
&\partial_1\otimes 1 - 1\otimes f^{-1}(\partial_1), \dots, \partial_n\otimes 1 - 1\otimes f^{-1}(\partial_n))\subset A^e. 
\end{align*}
\item\label{idn2}
If 
\begin{align*}
J = (&x_1\otimes 1 - 1\otimes q_1, \dots, x_n\otimes 1 - 1\otimes q_n, \\
&\partial_1\otimes 1 - 1\otimes p_1, \dots, \partial_n\otimes 1 - 1\otimes p_n)\subset A^e, 
\end{align*}
then $f$ is invertible with 
\begin{align*}
f^{-1}(x_1) &= q_1\\
&\vdots\\
f^{-1}(x_n) &= q_n\\
f^{-1}(\partial_1) &= p_1\\
&\vdots\\
f^{-1}(\partial_n) &= p_n
\end{align*}
\end{enumerate}
\end{Theorem}
\begin{proof}
Let us prove that
\begin{align}
\begin{split}\label{identideal}
\Ann_{A^e} 1_{A^f} = (&f(x_1)\otimes 1 - 1 \otimes x_1, \dots, 
f(x_n) \otimes 1 - 1 \otimes x_n, \\
&f(\partial_1)\otimes 1 - 1 \otimes \partial_1, \dots, 
f(\partial_n) \otimes 1 - 1 \otimes \partial_n).
\end{split}
\end{align}
If $a_1 f(b_1) + \cdots + a_m f(b_m) = 0$, for $a_1, b_1, \dots, a_m, b_m\in A$, then
$$
a_1\otimes b_1 + \cdots + a_m\otimes b_m = 
-(a_1\otimes 1)(f(b_1)\otimes 1 - 1 \otimes b_1) - \cdots - (a_m\otimes 1)(f(b_m)\otimes 1 - 1 \otimes b_m). 
$$
This shows that $\Ann_{A^e} 1_{A^f}$ is generated by $\{f(a) \otimes 1 - 1 \otimes a\mid a\in A\}$. The identity
\begin{align*}
f(a b) \otimes 1 - 1 \otimes (a b) &= f(a) f(b) \otimes 1 - 1 \otimes (a b)\\
&= (f(a) \otimes 1)(f(b)\otimes 1 - 1 \otimes b) + (1\otimes b) (f(a) \otimes 1 - 1 \otimes a)
\end{align*}
is verified using $A^e = A\otimes_k A^\op$ and shows \eqref{identideal} inductively as $x_1, \dots, x_n, \partial_1, \dots, \partial_n$ generate $A$.
An analogous argument gives
\begin{align*}
\Ann_{A^e} 1_{{}^{f^{-1}}\kern-4pt A} = (&x_1\otimes 1 - 1 \otimes f^{-1}(x_1), \dots, 
x_n \otimes 1 - 1 \otimes f^{-1}(x_n), \\
&\partial_1\otimes 1 - 1 \otimes f^{-1}(\partial_1), \dots, 
\partial_n \otimes 1 - 1 \otimes f^{-1}(\partial_n)).
\end{align*}
Therefore \eqref{idn1} is a consequence of \eqref{graph:ann} in Proposition \ref{Proposition:graphinv}. 

Suppose now that 
\begin{align*}
J = \Ann_{A^e} 1_{A^f} =(&x_1\otimes 1 - 1\otimes q_1, \dots, x_n\otimes 1 - 1\otimes q_n, \\
&\partial_1\otimes 1 - 1\otimes p_1, \dots, \partial_n\otimes 1 - 1\otimes p_n)\subset A^e.
\end{align*}
Then 
\begin{align*}
x_1 &= f(q_1)\\
&\vdots\\
x_n &= f(q_n)\\
\partial_1 &= f(p_1)\\
&\vdots\\
\partial_n &= f(p_n)
\end{align*}
 and \eqref{idn2} follows from Lemma \ref{Lemma:EndSur}.
\end{proof}

\subsection{The Gr\"obner basis in $A_{2n}(k)$}

We may translate the results of Theorem \ref{Theorem:Ae} into statements
about Gr\"obner basis in $A_{2n}(k)\cong A^e$ using the isomorphisms in
\eqref{double} and \eqref{transp}. This involves viewing $f$ as an
endomorphism of the Weyl algebra in the variables $x_1, \dots, x_n, 
\partial_1, \dots, \partial_n$ and the potential inverse as an endomorphism 
of the Weyl algebra in the
variables $x_{n+1}, \dots, x_{2n}, \partial_{n+1}, \dots, \partial_{2n}$.

\begin{Theorem}\label{Theorem:A2n}
Let $k$ be an arbitrary field and $A = A_n(k)$. For an endomorphism
$$
f: A \rightarrow A,
$$
we define the left ideal
\begin{equation}\label{GBideal}
J = (x_{n+1} - f(x_1), \dots, x_{2n} - f(x_n),
\partial_{n+1} + f(\partial_1), \dots, \partial_{2n} + f(\partial_n))
\end{equation}
in $A_{2n}(k)$. 
\begin{enumerate}[(i)]
\item\label{id1}
If $f$ is invertible, then
%\begin{equation*}
%J = (x_1 - f^{-1}(x_{n+1}), \dots, x_n - f^{-1}(x_{2n}), 
%\partial_1 + f^{-1}(\partial_{n+1}), \dots, \partial_n +  f^{-1}(\partial_{2n})). 
%\end{equation*}
%\item The above identity is wrong! It should read
$$
J = (x_1 - \tau f^{-1}(x_{n+1}), \dots, x_n - \tau f^{-1}(x_{2n}), 
\partial_1 - \tau f^{-1}(\partial_{n+1}), \dots, \partial_n - \tau  f^{-1}(\partial_{2n})),
$$
where $\tau$ is the antihomomorphism given by $\tau(x_{n+i}) = x_{n+i}$ and
$\tau(\partial_{n+i}) = - \partial_{n+i}$.

\item\label{id2}
% If
% \begin{equation}\label{GBlook}
% J = (x_1- q_1, \dots, x_n- q_n, 
% \partial_1 + p_1, \dots, \partial_n + p_n)
% \end{equation}
% with $p_1, \dots, p_n, q_1, \dots, q_n$ in the subalgebra generated by
% $x_{n+1}, \dots, x_{2n}, \partial_{n+1}, \dots, \partial_{2n}$,
% then $f$ is invertible with 
% \begin{align*}
% f^{-1}(x_{n+1}) &= q_1\\
% &\vdots\\
% f^{-1}(x_{2n}) &= q_n\\
% f^{-1}(\partial_{n+1}) &= p_1\\
% &\vdots\\
% f^{-1}(\partial_{2n}) &= p_n.
% \end{align*}
% \item The above is wrong! It should read
If
\begin{equation}\label{GBlook}
J = (x_1- q_1, \dots, x_n- q_n, 
\partial_1 - p_1, \dots, \partial_n - p_n)
\end{equation}
with $p_1, \dots, p_n, q_1, \dots, q_n$ in the subalgebra generated by
$x_{n+1}, \dots, x_{2n}, \partial_{n+1}, \dots, \partial_{2n}$,
then $f$ is invertible with 
\begin{align*}
f^{-1}(x_{n+1}) &= \tau q_1\\
&\vdots\\
f^{-1}(x_{2n}) &= \tau q_n\\
f^{-1}(\partial_{n+1}) &= \tau p_1\\
&\vdots\\
f^{-1}(\partial_{2n}) &= \tau p_n.
\end{align*}
\item
Let $\prec$ denote the lexicographic term order on monomials in $A_{2n}(k)$ given by 
\begin{equation}\label{lex}
x_1 \succ \cdots \succ x_n \succ \partial_1 \succ \cdots \succ \partial_n\succ x_{n+1} \succ \cdots
\succ x_{2n} \succ \partial_{n+1}\succ \cdots \succ \partial_{2n}.
\end{equation}
Then $f$ is invertible if and only if the reduced Gr\"obner basis of $J$ has the form
\begin{equation*}
J = (x_1- q_1, \dots, x_n- q_n, 
\partial_1 + p_1, \dots, \partial_n + p_n)
\end{equation*}
with $p_1, \dots, p_n, q_1, \dots, q_n$ in the subalgebra generated by
$x_{n+1}, \dots, x_{2n}, \partial_{n+1}, \dots, \partial_{2n}$,
\end{enumerate}
\end{Theorem} 

A perhaps surprising consequence of \eqref{right2} in 
Theorem \ref{Theorem:right} 
is the following result.

\begin{Theorem}\label{Theorem:surpr}
Let $k$ be a field of characteristic zero and
$A_{2n}(k)$ the Weyl algebra in the
variables
$x_1, \dots, x_{2n}, \partial_1, \dots, \partial_{2n}$.
If $p_1, \dots, p_n, q_1, \dots, q_n$ are in the subalgebra
generated by $x_1, \dots, x_n, \partial_1, \dots, \partial_n$,
then the left ideal
\begin{equation}\label{GBideal2}
J = (x_{n+1} - q_1, \dots, x_{2n} - q_n,
\partial_{n+1} + p_1, \dots, \partial_{2n} + p_n)
\end{equation}
is proper if and only if
\begin{align}
\begin{split}\label{pqcomms}
[q_i, q_j] &= 0\\
[p_i, p_j] &= 0\\
[p_i, q_j] &= \delta_{ij}.
\end{split}
\end{align}
\end{Theorem}
\begin{proof}
If $J$ is a proper ideal, then $A_{2n}/J$ can be viewed as a
non-zero bimodule over $A_n(k)$, which is cyclic generated by 
$1\in A_{2n}(k)$ as a left module over $A_n(k)$ identified with
the subalgebra generated by $x_1, \dots, x_n, \partial_1, \dots, 
\partial_n$.  The commutators in 
\eqref{pqcomms} are all elements of the left ideal $J$ except for
$[p_i, q_i]$. Here 
$$
[\partial_{n+i} + p_i, x_{n+i} - q_i] = 1 - [p_i, q_i]\in J.
$$
Part \eqref{right2} of Theorem \ref{Theorem:right} implies that 
$J$ cannot contain elements of positive degree from 
the subalgebra generated by $x_1, \dots, x_n, \partial_1, \dots, 
\partial_n$. Therefore the commutator identities in \eqref{pqcomms}
must hold if $J$ is a proper ideal. 

If the commutator identities in \eqref{pqcomms} hold, then the generators of $J$
in \eqref{GBideal2} form the unique reduced Gr\"obner basis with respect to the
term order in Theorem \eqref{lex}, 
since they are commuting with 
relatively prime leading monomials (cf.~Theorem \ref{Theorem:GBlmprime}).
This shows that $J$ is proper.
\end{proof}

\begin{Remark}
Theorem \ref{Theorem:surpr} fails for a field $k$ of positive characteristic $p>0$. 
Consider the left $A_2(k)$-module given by
$$
M = k[x_1, x_2]/(x_1^p, x_2^p)
$$
as a quotient of the natural left $A_2(k)$-module $k[x_1, x_2]$.
The annihilator of the element $[(x_2 - x_1)^{p-1}]\in M$ contains the left ideal
$$
J = (x_2 - x_1, \partial_2 + \partial_1 + \partial_1^{p+1}) \subsetneq A_2(k),
$$ 
but $[\partial_1 + \partial_1^{p+1}, x_1] = 1 + \partial_1^p$. 
\end{Remark}

\newcommand{\germ}{\mathfrak}

\providecommand{\bysame}{\leavevmode\hbox to3em{\hrulefill}\thinspace}
\providecommand{\MR}{\relax\ifhmode\unskip\space\fi MR }
\providecommand{\MRhref}[2]{%
  \href{http://www.ams.org/mathscinet-getitem?mr=#1}{#2}
}
\providecommand{\href}[2]{#2}

%\newcommand{\germ}{\mathfrak}

%\bibliographystyle{amsplain}
%\bibliography{graph}

\end{document}